\documentclass[12pt,a4j]{article}

\input{satoshimatome1102.STY}
\begin{document}
\title{Nagata embedding and $\scr{A}$-schemes}
\author{Satoshi Takagi
}
\date{July 19, 2011}
\maketitle

\begin{abstract}
We define the notion of normal $\scr{A}$-schemes,
and approximable $\scr{A}$-schemes.
Approximable $\scr{A}$-schemes inherit
many good properties of ordinary schemes.
As a consequence, we see that the Zariski-Riemann space
can be regarded in two ways -- either as the limit space
of admissible blow ups, or as the universal compactification
of a given non-proper scheme.
We can prove Nagata embedding
using Zariski-Riemann spaces.
\end{abstract}

\tableofcontents


\setcounter{section}{-1}
\section{Introduction}
We introduced the concept of $\scr{A}$-schemes
in \cite{TakZR}. In this paper, we will investigate further
properties of $\scr{A}$-schemes, mainly focusing
on Zariski-Riemann spaces.

First, we will show that there is a normalization
of $\scr{A}$-schemes, just as for ordinary schemes.
This is important, since we are aiming for an analog
of Zariski's main theorem.

One of the advantage of introducing $\scr{A}$-schemes
is that we can simplify the proof of Nagata embedding theorem;
it can be proven intuitively,
as in the original paper of Nagata \cite{Nagata}.
Note that the essential part is already proven
in Corollary 4.4.6 of \cite{TakZR}.
Compare with the proof of Conrad \cite{Conrad},
which only uses ordinary schemes, but
is long (approximately 50 pages).

Also, we introduce the notion of approximable $\scr{A}$-schemes:
an $\scr{A}$-scheme is approximable if it is a (filtered)
projective limit of ordinary schemes.
This notion is convenient, since locally free sheaves
on approximable schemes always come from a pull
back of a localy free sheaves on ordinary schemes.
At the same time, 
we see that the Zariski-Riemann space defined in \cite{TakZR}
is identified with the conventional one, namely the
limit space of $U$-admissible blow ups
along the exceptional locus $X \setminus U$
where $U$ is an open subscheme of a scheme $X$.
This shows that the conventional Zariski-Riemann space
has the desirable universal property in the category
of $\scr{A}$-schemes, not only with schemes.

This paper is organized as follows.
In section 1, we quickly review the definitions
and properties of $\scr{A}$-schemes,
which plays the central role in this paper.
In section 2, we construct the normalization functor
of $\scr{A}$-schemes.
In section 3, we give a notion of approximable
$\scr{A}$-schemes. This actually determines
the complete hull of the category $\cat{$\scr{Q}$-Sch}$
of ordinary schemes in the category $\cat{$\scr{A}$-Sch}$
of $\scr{A}$-schemes, namely the smallest complete
full subcategory of $\cat{$\scr{A}$-Sch}$ containing $\cat{$\scr{Q}$-Sch}$.
In section 4, we will give a proof of the original
version of Nagata embedding, which says
that any separated scheme of finite type 
can be embedded as an open subscheme
of a proper scheme.

\textbf{Notation and conventions.}
In this paper, the algebraic type $\scr{A}$
is always that of rings when we talk of $\scr{A}$-schemes.
When we say ordinary schemes,
we treat only coherent schemes
and quasi-compact morphisms between them;
to emphasize this assumption and 
to distinguish ordinary coherent schemes from $\scr{A}$-schemes,
we will say $\scr{Q}$-schemes instead of coherent schemes.

For an $\scr{A}$-scheme $X$,
the description $|X|$ stands for the underlying
topological space, which is coherent.

An $\scr{A}$-scheme $X$ will be called \textit{integral}
if it is irreducible and reduced.
This condition is in fact, stronger than
assuming any section ring $\scr{O}_{X}(U)$
is integral.

A morphism of $\scr{A}$-schemes
is \textit{proper}, if it is separated and universally closed.
We \textit{do not include the condition ``of finite type"}.

An open covering of an $\scr{A}$ scheme $X$
is denoted by $\{\amalg U_{ijk} \rightrightarrows  \amalg U_{i}\}$;
here, $\{U_{i}\}$ is a quasi-compact open covering of $X$,
and $\{U_{ijk}\}$ is a quasi-compact open covering 
of $U_{i} \cap U_{j}$ for each $i,j$.
Hence, there are open immersions $U_{ijk} \to U_{i}$
and $U_{ijk} \to U_{j}$ for each $i,j,k$.
Moreover, if $X$ is a $\scr{Q}$-scheme, we usually take
$U_{i}$ and $U_{ijk}$ as open affine subschemes of $X$.

\section{A brief review of $\scr{A}$-schemes}

In this section, we will recall some terminologies
and definitions in \cite{TakZR}.
A good reference for general lattice theories
is \cite{Stone}.

A topological space $X$ is \textit{coherent},
if it is sober, quasi-compact, quasi-separated,
and has a quasi-compact open basis.

The category $\cat{Coh}$ of coherent spaces
and quasi-compact morphisms is isomorphic
to the opposite category $\cat{DLat}^{\op}$ of distributive lattices
by the functor $C(-)_{\cpt}$;
for a coherent space $X$,
we may regard $C(X)_{\cpt}$ as the set of quasi-compact
open subsets of $X$, or the set of their complements.
A quasi-compact morphism $f:X \to Y$ induces
a morphism $f^{-1}:C(Y)_{\cpt} \to C(X)_{\cpt}$
of lattices.

Therefore, for any algebraic type $\scr{A}$,
we can regard an $\scr{A}$-valued sheaf on a coherent
space $X$ as a continuous covariant functor
$C(X)_{\cpt} \to \scr{A}$.

On a coherent space $X$,
there is a canonical $\cat{DLat}$-valued sheaf $\tau_{X}$
on $X$, which is defined by $U \mapsto C(U)_{\cpt}$
for quasi-compact open $U$;
this extends uniquely to the entire Zariski site of $X$.

We have a functor $\alpha_{1}: \cat{Rng} \to \cat{Dlat}$
from the category of commutative rings,
which sends a ring $R$ to the set of finitely generated ideals of $R$
modulo the relation $I^{2}=I$.
Note that this gives the usual spectrum of rings,
when combined with the previous isomorphism $C(-)_{\cpt}$.

Also, we have a natural homomorphism
$R \to \alpha_{1}(R)$ of multiplicative monoids,
sending $a \in R$ to the principal ideal generated by $a$.
This homomorphism commutes with localizations.

An $\scr{A}$-scheme is a triple $X=(|X|, \scr{O}_{X},\beta_{X})$,
where
\begin{enumerate}[(i)]
\item $|X|$ is a coherent space (the ``underlying space"),
\item $\scr{O}_{X}$ is a ring-valued sheaf on $|X|$
(the ``structure sheaf"), and 
\item $\beta_{X}:\alpha_{1}\scr{O}_{X} \to \tau_{X}$
is a morphism of $\cat{DLat}$-valued sheaves
(the ``support morphism"),
\end{enumerate}
which satisfies the following condition:
for an inclusion $V \hookrightarrow U$ of quasi-compact
open subsets of $|X|$, the restriction map
$\scr{O}_{X}(U) \to \scr{O}_{X}(V)$
factors through $\scr{O}_{X}(U)_{Z}$,
where $\scr{O}_{X}(U)_{Z}$ is the localization
along the multiplicative system
\[
\{a \in \scr{O}_{X}(U) \mid \beta_{X}(a) \subset Z\},
\]
where $Z=U \setminus V$ is the complement closed subset of $V$
in $U$.
By this property, $\scr{A}$-schemes are
locally ringed spaces.
A morphism of $\scr{A}$-schemes $f=(f,f^{\#}): X \to Y$
is a morphism of ringed spaces, which commutes 
with the support morphism:
\[
\xymatrix{
\alpha_{1}\scr{O}_{Y} \ar[r]^{\alpha_{1}f^{\#}} \ar[d]_{\beta_{Y}} &
f_{*}\alpha_{1}\scr{O}_{X} \ar[d]^{\beta_{X}} \\
\tau_{Y} \ar[r]_{f^{-1}} & f_{*}\tau_{X}
}
\]
The category $\cat{$\scr{A}$-Sch}$ of $\scr{A}$-schemes
is complete, and co-complete.

We have a fully faithful functor
$\cat{$\scr{Q}$-Sch} \to \cat{$\scr{A}$-Sch}$,
which preserves pull backs and finite patchings by quasi-compact opens.

Let $\Spec K \to S$ be a dominant morphism
of $\scr{A}$-schemes, where $S$ is integral and $K$ is a field.
The \textit{Zariski-Riemann space} $\ZR^{f}(K,S)$
is a proper $\scr{A}$-scheme over $S$,
defined as follows:
the points of the underlying space $|\ZR^{f}(K,S)|$
corresponds to the set of 
dominant morphisms $\scr{O}_{S,s} \to R$
of local rings, where $R$'s are valuation rings of $K$,
and $\scr{O}_{S,s}$' are local rings of $S$.
The map $|\ZR^{f}(K,S)| \to |S|$ is defined naturally.
The topology of $|\ZR^{f}(K,S)|$ is generated by
the \textit{domains} $\{R \in |\ZR^{f}(K,S)| \mid a \in R\}$
for $a \in K \setminus \{0\}$ and
the inverse images of the open subsets of $S$.
For a quasi-compact open subset $U$ of $|\ZR^{f}(K,S)|$,
$\scr{O}_{\ZR^{f}(K,S)}(U)$ is the ring
which is the intersection of all the valuation rings corresponding
to the points in $U$.

For arbitrary dominant morphism $X \to S$
of integral $\scr{A}$-schemes,
$\ZR^{f}(X,S)$ is defined by the pushout
of $X \leftarrow \ZR^{f}(K,X) \to \ZR^{f}(K,S)$,
where $K=Q(X)$ is the rational function field of $X$.
$X$ is \textit{of profinite type}
(resp. \textit{strongly of profinite type}) over $S$,
if the canonical map $\ZR^{f}(K,X) \to \ZR^{f}(K,S)$
is a scheme-theoretic immersion (resp. open immersion).

Here, we would like to list up some important properties
which is going to be used in  the sequel:
\begin{enumerate}
\item
A separated, of finite type morphism of $\scr{Q}$-schemes is always
strongly of profinite type.
\item The map $X \mapsto \ZR^{f}(X,S)$
gives the universal proper $\scr{A}$-scheme of profinite type over $S$
for each integral $\scr{A}$-scheme $X$ of profinite type over $S$.
\end{enumerate}

\section{Normalization}
In this section,
we fix an integral base $\scr{A}$-scheme $S$,
and any $\scr{A}$-scheme is integral, of profinite type over $S$.
We denote by $\cat{Int. $\scr{A}$-Sch}$
the category of integral $\scr{A}$-schemes
of profinite type over $S$ and dominant morphisms.

\begin{Def}
An $\scr{A}$-scheme $X$ is \textit{normal},
if the ring of every stalk $\scr{O}_{X,x}$
is integrally closed.
\end{Def}

\begin{Rmk}
We do not assume Noetherian property on
normal rings (or schemes) in this paper.
\end{Rmk}

\begin{Thm}
Let $\cat{N. $\scr{A}$-Sch}$
be the full subcategory of $\cat{Int. $\scr{A}$-Sch}$,
consisting of normal schemes,
and $U:\cat{N. $\scr{A}$-Sch} \to \cat{Int. $\scr{A}$-Sch}$
be the underlying functor.
Then, $U$ has a right adjoint `$\nor$'.
Moreover, the counit $\eta: U\circ \nor \Rightarrow \Id$
is proper dominant.
\end{Thm}
We will refer to this right adjoint as the \textit{normalization functor}.
\begin{proof}
The proof is somewhat long,
so we will divide it into several steps.
The construction of the normalization functor
is analogous to that of Zariski-Riemann spaces,
described in detail in \cite{TakZR}.
We will denote by $R^{\nor}$ the integral closure
of a given integral domain $R$ in the sequel.
\begin{itemize}
\item[Step 1:] First,
we will construct the underlying space of 
the normalization of a given integral $\scr{A}$-scheme $X$.
Let $\scr{N}^{X}_{0}$ be the set of finite sets of pairs
$(U,\alpha)$, where
\begin{enumerate}[(a)]
\item $U$ is a quasi-compact open subset of $X$, and
\item $\alpha \in \scr{O}_{X}(U)^{\nor}\setminus \{0\}$.
\end{enumerate}
Let $\mathfrak{a}=\{(U_{i},\alpha_{i})\}_{i}$,
$\mathfrak{b}=\{(V_{j},\beta_{j})\}_{j}$
be two elements of $\scr{N}^{X}_{0}$.
We define two operations $+$, $\cdot$
on $\scr{N}^{X}_{0}$ by
\[
\mathfrak{a}+\mathfrak{b}=\mathfrak{a} \cup \mathfrak{b},
\mathfrak{a}\cdot \mathfrak{b}
=\{ (U_{i} \cap V_{j},\alpha_{i}\beta_{j})\}_{ij}
\]
For a pair $(U,\alpha)$,
define $U[\alpha]$ as
\[
U[\alpha]=\{x \in U \mid \rom{$x$ is in the image of 
$\Spec \scr{O}_{X,x}[\alpha^{-1}]\to \Spec \scr{O}_{X,x}$}\}.
\]
For two elements $\mathfrak{a}=\{(U_{i},\alpha_{i})\}_{i}$,
$\mathfrak{b}=\{(V_{j},\beta_{j})\}_{j}$,
the relation $\mathfrak{a} \prec \mathfrak{b}$ holds if
\begin{enumerate}[(a)]
\item $U_{i}[\alpha_{i}] \subset \cup_{j}V_{j}[\beta_{j}]$
for any $i$, and
\item For any $x \in U_{i}[\alpha_{i}]$,
set $J_{x}=\{j \mid x \in V_{j}[\beta_{j}]\}$.
Then $(\beta_{j})_{j \in J_{x}}$ generates the unit ideal in
$\scr{O}_{X,x}^{\nor}[\alpha_{i}^{-1}]$.
\end{enumerate}
Let $\approx$ be the equivalence relation
generated by $\prec$, and set
$\scr{N}^{X}=\scr{N}^{X}_{0}/\approx$.
The addition and multiplication of $\scr{N}^{X}_{0}$
descends to $\scr{N}^{X}$,
which makes $\scr{N}^{X}$ into a distributive lattice.
Set $|X^{\nor}|=\Spec \scr{N}^{X}$.
This is the underlying space of the normalization $X^{\nor}$.
\item[Step 2:]
There is a natural homomorphism $C(X)_{\cpt} \to \scr{N}^{X}$
of distributive lattices, defined by $Z \mapsto \{(Z,1)\}$.
This defines a quasi-compact morphism
$\pi:|X^{\nor}| \to |X|$ of coherent spaces.
\item[Step 3:]
Let $p$ be a point of $|X^{\nor}|$,
and set $x=\pi(p)$.
Then,
\[
\mathfrak{p}=\{ a \in \scr{O}_{X,x}^{\nor} \mid (X,a) \leq p\}
\]
becomes a prime ideal of $\scr{O}_{X,x}^{\nor}$.
Let $R_{p}$ be the localization of $\scr{O}_{X,x}^{\nor}$
by $\mathfrak{p}$.
Then, $R_{p}$ dominates $\scr{O}_{X,x}$.
\item[Step 4:]
The structure sheaf $\scr{O}_{X^{\nor}}$ is defined by
\[
U \mapsto \{ a \in K \mid a \in R_{p} \quad (p \in U)\},
\]
where $K$ is the function field of $X$.
The support morphism $\beta_{X^{\nor}}:
\alpha_{1}\scr{O}_{X^{\nor}}\to \tau_{X^{\nor}}$
is defined by
\[
\alpha_{1}\scr{O}_{X^{\nor}}(U)
\ni (a_{1},\cdots, a_{n}) \mapsto
\{(U,a_{i})\}_{i}.
\]
This defines an $\scr{A}$-scheme
$X^{\nor}=(|X^{\nor}|,\scr{O}_{X^{\nor}},\beta_{X^{\nor}})$.
\item[Step 5:] We have a canonical morphism
of sheaves $\scr{O}_{X} \to \pi_{*}\scr{O}_{X^{\nor}}$,
defined by the identity $a \mapsto a$.
This yields a morphism $\pi:X^{\nor} \to X$ of $\scr{A}$-schemes.
It is of profinite type, by the criterion 4.3.3 in \cite{TakZR}.
\item[Step 6:] Let us show that $\pi$ is proper.

We can see from the construction
that we have a natural morphism $\ZR^{f}(K,X) \to X^{\nor}$:
the morphism $|\ZR^{f}(K,X)| \to |X^{\nor}|$ of underlying spaces
is defined by
\[
\scr{N}^{X} \to \scr{M}^{X} \quad
(\{ (U_{i},\alpha_{i})\}_{i}
 \mapsto \{(X \setminus U_{i},\{\alpha_{i}^{-1}\})\}_{i} ),
\]
where $\scr{M}^{X}=C(\ZR^{f}(K,X))_{\cpt}$,
and the morphism between the structure sheaves is canonical.
Note that $\ZR^{f}(K,X)$ is already normal.
This shows that $X^{\nor}$ is proper over $X$
by the valuative criterion.

\item[Step 7:] We will show that the normalization is a functor.
Let $f:X \to Y$ be a dominant morphism of $\scr{A}$-schemes.
$|f^{\nor}|:X^{\nor} \to Y^{\nor}$ is defined by
\[
\scr{N}^{Y} \to \scr{N}^{X}:
\{ (U_{i},\alpha_{i})\}_{i} \mapsto \{ (f^{-1}U_{i},f^{\#}\alpha_{i})\}_{i}.
\]
The morphism $f^{\#}:\scr{O}_{Y} \to f_{*}\scr{O}_{X}$
extends canonically to $f^{\nor,\#}:\scr{O}_{Y^{\nor}}
\to f^{\nor}_{*}\scr{O}_{X^{\nor}}$.
This gives a functor
$\nor:\cat{Int. $\scr{A}$-Sch} \to \cat{N. $\scr{A}$-Sch}$.

\item[Step 8:] It remains to show that
the normalization functor is indeed the right adjoint
of the underlying functor.
The unit $\epsilon:\Id \Rightarrow \nor \circ U$ is the identity,
since the normalization of a normal $\scr{A}$-scheme
is trivial.
The counit $\eta: U \circ \nor \Rightarrow \Id$ is given
by $\pi$ defined above.
\end{itemize}
\end{proof}
\begin{Rmk}
We know that the integral closure $\tilde{R}$ of a Noetherian domain
$R$ is not Noetherian \cite{Nagata2}.
Therefore, we must drop the `of finite type' condition
from the definition of properness, if we wish to
say that ``the normalization $\Spec \tilde{R}\to \Spec R$ is proper".
\end{Rmk}

\begin{Lem}
\label{lem:normal:local}
Let $X$ be a normal $\scr{A}$-scheme.
Then. $\scr{O}_{X}(U)$ is normal for any open $U$.
\end{Lem}
\begin{proof}
Let $b \in K$ be an element which
is integral over $\scr{O}_{X}(U)$,
 where $K$ is the function field of $X$.
Since $b_{x}$ is integral over the stalk $\scr{O}_{X,x}$
for any $x \in U$ and $\scr{O}_{X,x}$ is integrally closed,
we have $b_{x} \in \scr{O}_{X,x}$.
Hence, $b \in \scr{O}_{X}(U)$.
\end{proof}

\begin{Prop}
The normalization functor
coincides with the usual normalization,
when restricted to $\scr{Q}$-schemes.
\end{Prop}
\begin{proof}
First, we will show for affine schemes $X=\Spec A$.
The universality of the normalization functor
gives a canonical morphism $f:\Spec (A^{\nor}) \to X^{\nor}$.
Since $\Gamma(X^{\nor},\scr{O}_{X,x})$
is normal, we have a canonical homomorphism 
$A^{\nor} \to \Gamma(X^{\nor},\scr{O}_{X,x})$.
This yields a morphism $g: X^{\nor} \to \Spec (A^{\nor})$.
It is easy to check that these two morphisms $f,g$
are inverse to each other.

It is obvious from the construction that
normalization commutes with localizations.
This shows that the normalization
of any $\scr{Q}$-scheme coincides
with the usual definition.
\end{proof}

\section{Approximations by ordinary schemes}

We fix an integral base $\scr{Q}$-scheme $S$
in the sequel.
The next proposition is pure category-theoretical
and easy, so we will omit the proof.
\begin{Prop}
Let $\mathcal{B},\mathcal{C}$ be two categories,
with $\mathcal{B}$ finite complete
and $\mathcal{C}$ small complete.
Let $F:\mathcal{B} \to \mathcal{C}$
be a finite continuous functor, namely
$F$ preserves fiber products.
For any object $a$ of $\mathcal{C}$,
The followings are equivalent:
\begin{enumerate}[(i)]
\item $a$ is isomorphic to a limit of the objects in $\Imag F$.
\item $a$ is isomorphic to a filtered limit of the objects in $\Imag F$.
\end{enumerate}
\end{Prop}

\begin{Def}
Let $X$ be an $\scr{A}$-scheme,
and $\scr{P}$ be a class of $\scr{Q}$-schemes.
\begin{enumerate}
\item $X$ is \textit{approximable by $\scr{P}$},
if $X$ is isomorphic to a filtered limit of some objects of $\scr{P}$.
\item $X$ is \textit{approximable},
if $X$ is isomorphic to a filtered limit of some $\scr{Q}$-schemes.
\end{enumerate}
\end{Def}

\begin{Prop}
Any approximable $\scr{A}$-scheme
is approximable by $\scr{Q}$-schemes of finite type.
\end{Prop}
\begin{proof}
It suffices to show that
any $\scr{Q}$-scheme is approximable 
by $\scr{Q}$-schemes of finite type.

Let $X$ be any $\scr{Q}$-scheme,
and $\{\amalg U_{ijk} \to \amalg U_{i}\}$ be a finite affine covering of $X$.
Since $U_{ijk} \to U_{i}$
is quasi-compact, $U_{ijk}$ is of finite type over $U_{i}$.
Thus, we have approximations
$U_{i}=\varprojlim_{\lambda}U^{\lambda}_{i}$ and
$U_{ijk}=\varprojlim_{\lambda}U^{\lambda}_{ijk}$
so that $U_{i}^{\lambda}$ and $U_{ijk}^{\lambda}$
are of finite type and $U_{ijk}^{\lambda}\to U_{i}^{\lambda}$
are open immersions. We may also assume that the above limits
are filtered. Since filtered limits and finite colimits commute,
we have
\[
X=\varinjlim_{i}U_{i}
=\varinjlim_{i}\varprojlim_{\lambda}U_{i}^{\lambda}
=\varprojlim_{\lambda}\varinjlim_{i}U_{i}^{\lambda}
\]
and $\varinjlim_{i}U_{i}^{\lambda}$
is a $\scr{Q}$-scheme of finite type.
\end{proof}

\begin{Def}
Let $X,Y$ be two integral $\scr{A}$-schemes.
A morphism $f:X \to Y$ is \textit{birational},
if $f$ induces an isomorphism $Q(X) \simeq Q(Y)$
between the rational function fields.
\end{Def}
\begin{Rmk}
Note that, the morphism being birational does not imply the
existence of an open dense subset $U$ of $X$
such that $U \simeq f(U)$.
\end{Rmk}

\begin{Prop}
Let $X$ be an approximable $\scr{A}$-scheme,
say $X=\varprojlim_{\lambda}X^{\lambda}$
where $X^{\lambda}$'s are $\scr{Q}$-schemes.
\begin{enumerate}
\item If $X$ is reduced,
then $X$ is approximable by reduced $\scr{Q}$-schemes.
\item If $X$ is integral,
then $X$ is approximable by integral $\scr{Q}$-schemes.
\item Further, if the rational function field $Q(X)$
is finitely generated over an integral base $\scr{Q}$-scheme,
then $X$ is approximable by integral $\scr{Q}$-schemes
birational to $X$.
\item If $X$ is normal,
then $X$ is approximable by normal $\scr{Q}$-schemes.
\item If $X$ is proper and approximable by
separated $\scr{Q}$-schemes, then
$X$ is approximable by proper (and of finite type) $\scr{Q}$-schemes.
\end{enumerate}
\end{Prop}
\begin{proof}
The proofs are all similar, so let us just see (1).

Since $X$ is reduced, $X \to X^{\lambda}$
factors through the reduced $\scr{Q}$-scheme
$(X^{\lambda})_{\red}$. This shows that
$X\simeq \varprojlim(X^{\lambda})_{\red}$.
\end{proof}

\begin{Prop}
Let $f:X \to Y$ be a morphism of $\scr{A}$-schemes over $S$,
with $X$ approximable and $Y$ a $\scr{Q}$-scheme, of finite type
over $S$. 
\begin{enumerate}
\item Suppose $X$ is a filtered projective limit
$\varprojlim_{\lambda}X^{\lambda}$ of $\scr{Q}$-schemes.
Then, $f$ factors through $X \to X^{\lambda}$
for some $\lambda$.
\item Furthermore, if $X$ is proper over $S$
and approximable by separated $\scr{Q}$-schemes, and 
$Y$ is separated over $S$,
then the above $X^{\lambda}$ can be chosen to be
a proper scheme over $Y$.
\end{enumerate}
\end{Prop}
\begin{proof}
\begin{enumerate}
\item We may assume that $Y$ is affine.
Since $Y$ is of finite type and $\Gamma(X, \scr{O}_{X})$
is a filtered colimit of $\Gamma(X_{\lambda}, \scr{O}_{X^{\lambda}})$,
$f$ factors through $X^{\lambda}$ for some $\lambda$.
\item By the above proposition,
we may assume that $X^{\lambda}$'s are proper over the base scheme
$S$. Since $Y$ is separated, the morphism $X^{\lambda} \to Y$ is proper.
\end{enumerate}
\end{proof}

\begin{Thm}
\label{thm:str:sheaf:direct}
Let $f:X \to Y$ be a proper birational morphism,
where
$X$ is an integral $\scr{A}$-scheme
approximable by separated $\scr{Q}$-schemes,
and $Y$ a normal $\scr{Q}$-scheme separated 
and of finite type over $S$.
Then, $f_{*}\scr{O}_{X}=\scr{O}_{Y}$.
\end{Thm}
\begin{proof}
The previous proposition shows that
$f$ factors through proper morphisms
$f_{\lambda}:X^{\lambda} \to Y$,
where $X=\varprojlim_{\lambda}X^{\lambda}$
and $\{X^{\lambda}\}$ is a filtered system of integral
$\scr{Q}$-schemes, proper birational and of finite type over $Y$.
Since $Y$ is normal,
the usual Zariski's main theorem tells
that $\scr{O}_{Y} \to (f_{\lambda})_{*}\scr{O}_{X^{\lambda}}$
is an isomorphism (Corollary III. 11.4 of \cite{Harts}),
and $f_{*}\scr{O}_{X}$ coincides with the right hand side,
since it is a colimit of $(f_{\lambda})_{*}\scr{O}_{X^{\lambda}}$'s.
\end{proof}

Since ``of profinite type" morphisms
are stable under taking limits,
approximable $\scr{A}$-schemes are necessarily
of profinite type over $S$.
\begin{Thm}
Let $X$ be a normal $\scr{A}$-scheme,
proper and of profinite type over the integral base
$\scr{Q}$-scheme $S$.
Assume that the rational function field
$Q(X)$ is finitely generated over $Q(S)$.
The followings are equivalent:
\begin{enumerate}[(i)]
\item $X$ is approximable by separated $\scr{Q}$-schemes.
\item Let $\mathcal{U}=\{\amalg U_{ijk} \rightrightarrows \amalg U_{i}\}$
be any quasi-compact open covering of $X$.
Then, there exists a refinement
$\amalg V_{ijk} \rightrightarrows \amalg V_{i}$ of $\mathcal{U}$
such that $\Spec \scr{O}_{X}(V_{ijk}) \to \Spec \scr{O}_{X}(V_{i})$
are open immersions.
\end{enumerate}
\end{Thm}
\begin{proof}
(i)$\Rightarrow$(ii):
$X$ can be written as a filtered limit
$X=\varprojlim_{\lambda}X^{\lambda}$,
where $X^{\lambda}$'s are normal $\scr{Q}$-schemes,
proper and of finite type over $S$.
Since the number of $U_{i}$'s and $U_{ijk}$'s
are finite, $U_{i}=\pi^{-1}\tilde{U}_{i}$,
$U_{ijk}=\pi^{-1}\tilde{U}_{ijk}$ for some $\pi:X \to X^{\lambda}$,
where $\tilde{U}_{i}$'s and $\tilde{U}_{ijk}$'s
are quasi-compact open subsets of $X^{\lambda}$.
Take any refinement $\{\amalg \tilde{V}_{ijk} \to \amalg \tilde{V}_{i}\}$
of $\tilde{\mathcal{U}}=
\{\amalg \tilde{U}_{ijk} \to \amalg \tilde{U}_{i}\}$,
by affine opens $\tilde{V}_{ijk}$ and $\tilde{V}_{i}$.
Set $V_{ijk}=\pi^{-1}\tilde{V}_{ijk}$
and $V_{i}=\pi^{-1}\tilde{V}_{i}$.
Since $V_{i} \to \tilde{V}_{i}$ is proper
and $\tilde{V}_{i}$ is normal, of finite type,
Theorem \ref{thm:str:sheaf:direct} implies
that $\scr{O}_{X}(V_{i})=\scr{O}_{X^{\lambda}}(\tilde{V}_{i})$.
This shows that $\Spec \scr{O}_{X}(V_{ijk}) \to \Spec \scr{O}_{X}(V_{i})$
are open immersions.

(ii)$\Rightarrow$(i):
For any covering
$\mathcal{U}=\{\amalg U_{ijk} \rightrightarrows \amalg U_{i}\}$
of $X$,
the refinement $\amalg V_{ijk} \rightrightarrows \amalg V_{i}$
gives open immersions
 $\Spec \scr{O}_{X}(V_{ijk}) \to \Spec \scr{O}_{X}(V_{i})$
which patches up to give a $\scr{Q}$-scheme $X(\mathcal{U})$
and the canonical morphism $\pi_{\mathcal{U}}:X \to X(\mathcal{U})$.
The covering $\mathcal{U}$ is a pull back
of a covering of $X(\mathcal{U})$,
and ditto for the elements of $\scr{O}_{X}(U_{i})$'s.
From this observation, we see that the induced morphism
$X \to \varprojlim_{\mathcal{U}}X(\mathcal{U})$
is an isomorphism.
It is clear from the construction that $X(\mathcal{U})$
is proper.
\end{proof}

\section{Another proof of Nagata embedding}

In the sequel, any $\scr{A}$-schemes are integral.
\begin{Def}
Let $S$ be a $\scr{Q}$-scheme,
and $X$ be a $\scr{Q}$-scheme over $S$.
We say that $X$ is \textit{compactifiable}
over $S$, if there is an open immersion $X \to Y$
where $Y$ is a $\scr{Q}$-scheme,
proper, of finite type over $S$.
\end{Def}

\begin{Prop}
Let $S$ be a $\scr{Q}$-scheme,
and $X$ be a $\scr{Q}$-scheme over $S$.
The followings are equivalent:
\begin{enumerate}[(i)]
\item $X$ is compactifiable over $S$.
\item $\ZR^{f}(X,S)$ is approximable by separated $\scr{Q}$-schemes,
and the natural map $X \to \ZR^{f}(X,S)$ is an open immersion.
\end{enumerate}
\end{Prop}
\begin{proof}
(i)$\Rightarrow$(ii):
There exists an open immersion $X \to Y$
into a $\scr{Q}$-scheme $Y$, proper of finite type
over $S$. This morphism factors through
$\ZR^{f}(X,S)$ by the universal property.
We will show that for any quasi-compact open subset
$U$ of $\ZR^{f}(X,S)$,
there exists a proper birational morphism
$Y^{\prime} \to Y$, such that $g^{-1}(V)=U$
for some quasi-compact open subset $V$ of $Y^{\prime}$,
where $g:\ZR^{f}(X,S) \to Y^{\prime}$ is 
the canonical extension of $f:\ZR^{f}(X,S) \to Y$:
\[
\xymatrix{
U \ar[r] \ar@{.>}[d] \ar@{}[dr]|{\Box}
& \ZR^{f}(X,S) \ar@{.>}[d]_{g} \ar[rd]^{f} &
X \ar[l] \ar[d] \\
V \ar@{.>}[r] & Y^{\prime} \ar@{.>}[r] & Y
}
\]
By the construction of $\ZR^{f}(X,S)$,
we may assume $U$ is of the form
$U(W, \alpha)$, where $W$ is a quasi-compact open
subset of $S$ and $\alpha$ is a finite subset of $K \setminus \{0\}$,
and 
\[
U(W,\alpha)=\pi^{-1}(W) \cap 
\{ p \in \ZR^{f}(X,S) \mid \alpha \subset \scr{O}_{\ZR^{f}(X,S),p}\}.
\]
Note that $f(U \cap X)$ is open in $Y$, since $X \to \ZR^{f}(X,S)$
is an open immersion.
Suppose $\alpha=\{a_{i}/b_{i}\}_{i}$,
where $a_{i},b_{i} \in \scr{O}_{Y}$ locally.
Let $Y^{\prime} \to Y$ be the blow up
along $(Y\setminus X) \cap \Supp(a_{i},b_{i})$.
Then, either $a_{i}/b_{i}$ or $b_{i}/a_{i}$
is in $\scr{O}_{Y^{\prime}}$ locally,
which shows that the domain of $a_{i}/b_{i}$
is open in $Y^{\prime}$.
This shows that $U$ is the pull back
of some $V$ by the morphism $g:\ZR^{f}(X,S) \to Y^{\prime}$.
Hence, $\ZR^{f}(X,S) \to \varprojlim_{\lambda} Y^{\lambda}$
becomes a homeomorphism on the underlying space,
where $Y^{\infty}=\varprojlim_{\lambda}Y^{\lambda}$
is the filtered projective limit of $X$-admissible blow-ups
of $Y$. A similar argument shows that
$\scr{O}_{Y^{\infty}} \to \scr{O}_{\ZR^{f}(X,S)}$
also becomes an isomorphism.
Note that $Y^{\lambda}$'s are separated over $S$,
since we only used blow-ups.

(ii)$\Rightarrow$(i):
The Zariski-Riemann space $\ZR^{f}(X,S)$
can be written as a form $\varprojlim_{\lambda}Y^{\lambda}$,
where $Y^{\lambda}$'s are proper, of finite type $\scr{Q}$-schemes.
Since $X \to \ZR^{f}(X,S)$ is an open immersion
and $X$ is quasi-compact, $X \to \ZR^{f}(X,S) \to Y^{\lambda}$
becomes an open immersion for some $\lambda$.
\end{proof}

Now, we are on the stage to give the proof of the Nagata embedding.
\begin{Thm}[Nagata]
Let $S$ be a $\scr{Q}$-scheme,
and $X$ be a $\scr{Q}$-scheme,
separated and of finite type over $S$.
Then, $X$ is compactifiable over $S$.
\end{Thm}
In this section, we will prove
this theorem for the essential case,
namely when $S$ and $X$ are integral.
This restriction is due to the fact that
we simply haven't established the theorem
of Zariski-Riemann spaces for non-integral schemes.

Since $X$ is quasi-compact,
and affine schemes of finite type over $S$
is obviously compactifiable,
it suffices to prove the following proposition:
\begin{Prop}
Let $V_{1}$ and $V_{2}$ be compactifiable
open sub-$\scr{Q}$-schemes of a $\scr{Q}$-scheme $X$
separated over $S$,
with $X=V_{1} \cup V_{2}$.
Then $X$ is also compactifiable.
\end{Prop}
\begin{proof}
Consider $\ZR^{f}(X,S)$.
Since $X$ is separated, of finite type over $S$,
the morphism $X \to \ZR^{f}(X,S)$ is an open
immersion by Corollary 4.4.6 of \cite{TakZR}.
Let $W_{1}$ (resp. $W_{2}$) be the complement
of the closure of $V_{2}\setminus V_{1}$
(resp. $V_{1} \setminus V_{2}$) in $\ZR^{f}(X,S)$.

We can see that $W_{1} \cap W_{2}=V_{1} \cap V_{2}$,
since the interior of the complement of $V_{1} \cup V_{2}$
is empty.
Next, we see that $W_{1} \cup W_{2}=\ZR^{f}(X,S)$.
For this, it suffices to show that 
$\overline{V_{2}\setminus V_{1}}\cap \overline{V_{1} \setminus V_{2}}
=\emptyset$.
Suppose there is a point $p$ in
$\overline{V_{2}\setminus V_{1}}\cap \overline{V_{1} \setminus V_{2}}$.
Since $V_{2} \setminus V_{1}$ and $V_{1} \setminus V_{2}$
are coherent subsets of $\ZR^{f}(X,S)$,
$p$ must be a specialization of some $x_{1} \in V_{2} \setminus V_{1}$
and $x_{2} \in V_{1} \setminus V_{2}$
by Corollary 1.2.8 of \cite{TakZR}.
Since $\ZR^{f}(K,S) \to \ZR^{f}(X,S)$ is surjective,
there are inverse images  $y_{i} \in \ZR^{f}(K,S)$ of $x_{i}$
such that $y_{i}$ specializes to $p$.
The points in $\ZR^{f}(K,S)$ are valuation rings,
hence $y_{2}$ must be the specialization of $y_{1}$,
or the converse. In either cases,
this contradicts to the fact that $x_{1}$ and
$x_{2}$ has no specialization-generalization relations.
This also shows that $W_{1}$ and $W_{2}$ are quasi-compact.
The morphism $p_{1}:\ZR^{f}(V_{1},S) \to \ZR^{f}(X,S)$
induces an isomorphism on $W_{1}$,
hence $W_{1}$ is approximable by $\scr{Q}$-morphisms
of finite type over $S$, ditto for $W_{2}$.

Take any $\scr{Q}$-model $Y_{i}$
of $W_{i}$ (namely, a morphism $\pi_{i}:W_{i} \to Y_{i}$
where $Y_{i}$ is a $\scr{Q}$-scheme)
such that the morphism $\pi_{i}$
induces an isomorphism on $V_{i}$. 
Then, $Y_{1}$ and $Y_{2}$ can be patched along
$\pi_{1}(W_{1} \cap W_{2}) \simeq \pi_{2}(W_{1} \cap W_{2})$
to obtain a $\scr{Q}$-scheme $Y$ of finite type,
and a surjective morphism $\ZR^{f}(X,S)=W_{1} \cup W_{2} \to Y$.
This shows that $Y$ is proper.
\end{proof}

\textit{Acknowledgements}:
The essential part of the proof of the Nagata embedding
is taught by Professor F. Kato.
I was partially supported by the Grant-in-Aid for Young
Scientists (B) \# 23740017.

\textsc{S. Takagi: Department of Mathematics, Faculty of Science,
Kyoto University, Kyoto, 606-8502, Japan}

\textit{E-mail address}: \texttt{takagi@math.kyoto-u.ac.jp}

\end{document}